\documentclass[12pt]{amsart}
\usepackage{verbatim,amssymb,latexsym}
\usepackage[all]{xy}

\newcommand{\A}{\mathbf{A}}

\newcommand{\U}{\mathbf{U}}

\newcommand{\fx}{\mathfrak{x}}

\newcommand{\Spec}{\operatorname{Spec}}

\renewcommand{\lim}{\varprojlim}

\newcounter{spec}
{\end{list}}%

\swapnumbers

\newtheorem{thm}{Theorem}[section]
\newtheorem{lemma}[thm]{Lemma}

\theoremstyle{definition}

\newtheorem{rk}[thm]{Remark}

\numberwithin{equation}{section}

\entrymodifiers={!!<0pt,0.7ex>+}

\catcode`\@=11
\def\opn#1#2{\def#1{\mathop{\kern0pt\fam0#2}\nolimits}} 
\def\underrightarrow{\mathpalette\underrightarrow@}
\def\underrightarrow@#1#2{\vtop{\ialign{$##$\cr
 \hfil#1#2\hfil\cr\noalign{\nointerlineskip}%
 #1{-}\mkern-6mu\cleaders\hbox{$#1\mkern-2mu{-}\mkern-2mu$}\hfill
 \mkern-6mu{\to}\cr}}} 

\def\underleftarrow{\mathpalette\underleftarrow@}
\def\underleftarrow@#1#2{\vtop{\ialign{$##$\cr
 \hfil#1#2\hfil\cr\noalign{\nointerlineskip}#1{\leftarrow}\mkern-6mu
 \cleaders\hbox{$#1\mkern-2mu{-}\mkern-2mu$}\hfill
 \mkern-6mu{-}\cr}}}
    \let\phi=\varphi
    \let\epsilon=\varepsilon  
\def\:{\colon}   
\let\oldtilde=\tilde
\def\tilde#1{\mathchoice{\widetilde{#1}}{\widetilde{#1}}%
{\indextil{#1}}{\oldtilde{#1}}}
\def\indextil#1{\lower2pt\hbox{$\textstyle{\oldtilde{\raise2pt%
\hbox{$\scriptstyle{#1}$}}}$}}
\def\pnt{{\raise1.1pt\hbox{$\textstyle.$}}}  
\let\amp@rs@nd@\relax
\newdimen\ex@
\newdimen\bigaw@
\newdimen\minaw@
\newdimen\minCDaw@  
\newif\ifCD@
\def\minCDarrowwidth#1{\minCDaw@#1}

\def\@CD{\def\A##1A##2A{\llap{$\vcenter{\hbox
 {$\scriptstyle##1$}}$}\Big\uparrow\rlap%
{$\vcenter{\hbox{$\scriptstyle##2$}}$}&&}%
\def\V##1V##2V{\llap{$\vcenter{\hbox
 {$\scriptstyle##1$}}$}\Big\downarrow\rlap%
{$\vcenter{\hbox{$\scriptstyle##2$}}$}&&}%
\def\={&\hskip.5em\mathrel
 {\vbox{\hrule width\minCDaw@\vskip3\ex@\hrule width
 \minCDaw@}}\hskip.5em&}%
\def\verteq{\Big\Vert&&}%
\def\noarr{&&}%
\def\vspace##1{\noalign{\vskip##1\relax}}%
\relax\let\amp@rs@nd@&\iffalse}\fi
 \CD@true\vcenter\bgroup\relax%
\let\\=\cr\iffalse}\fi\tabskip\z@skip\baselineskip20\ex@
 &\hfill$\m@th##$\hfill\cr}
\def\@endCD{\cr\egroup\egroup}
\def\>#1>#2>{\amp@rs@nd@\setbox\z@\hbox{$\scriptstyle
 \;{#1}\;\;$}\setbox\@ne\hbox{$\scriptstyle\;{#2}\;\;$}\setbox\tw@
 \hbox{$#2$}\ifCD@
 \global\bigaw@\minCDaw@\else\global\bigaw@\minaw@\fi
 \ifdim\wd\z@>\bigaw@\global\bigaw@\wd\z@\fi
 \ifdim\wd\@ne>\bigaw@\global\bigaw@\wd\@ne\fi
 \ifCD@\hskip.5em\fi
 \ifdim\wd\tw@>\z@
 \mathrel{\mathop{\hbox to\bigaw@{\rightarrowfill}}\limits^{#1}_{#2}}\else
 \mathrel{\mathop{\hbox to\bigaw@{\rightarrowfill}}\limits^{#1}}\fi
 \ifCD@\hskip.5em\fi\amp@rs@nd@}
\def\<#1<#2<{\amp@rs@nd@\setbox\z@\hbox{$\scriptstyle
 \;\;{#1}\;$}\setbox\@ne\hbox{$\scriptstyle\;\;{#2}\;$}\setbox\tw@
 \hbox{$#2$}\ifCD@
 \global\bigaw@\minCDaw@\else\global\bigaw@\minaw@\fi
 \ifdim\wd\z@>\bigaw@\global\bigaw@\wd\z@\fi
 \ifdim\wd\@ne>\bigaw@\global\bigaw@\wd\@ne\fi
 \ifCD@\hskip.5em\fi
 \ifdim\wd\tw@>\z@
 \mathrel{\mathop{\hbox to\bigaw@{\leftarrowfill}}\limits^{#1}_{#2}}\else
 \mathrel{\mathop{\hbox to\bigaw@{\leftarrowfill}}\limits^{#1}}\fi
 \ifCD@\hskip.5em\fi\amp@rs@nd@}
\def\@CDS{\def\A##1A##2A{\llap{$\vcenter{\hbox
 {$\scriptstyle##1$}}$}\Big\uparrow\rlap%
{$\vcenter{\hbox{$\scriptstyle##2$}}$}&}%
\def\V##1V##2V{\llap{$\vcenter{\hbox
 {$\scriptstyle##1$}}$}\Big\downarrow\rlap%
{$\vcenter{\hbox{$\scriptstyle##2$}}$}&}%
\def\={&\hskip.5em\mathrel
 {\vbox{\hrule width\minCDaw@\vskip3\ex@\hrule width
 \minCDaw@}}\hskip.5em&}
\def\verteq{\Big\Vert&}
\def\novarr{&}
\def\noharr{&&}
\def\SE##1E##2E{\slantedarrow(0,18)(4,-3){##1}{##2}&}
\def\SW##1W##2W{\slantedarrow(24,18)(-4,-3){##1}{##2}&}
\def\NE##1E##2E{\slantedarrow(0,0)(4,3){##1}{##2}&}
\def\NW##1W##2W{\slantedarrow(24,0)(-4,3){##1}{##2}&}
\def\slantedarrow(##1)(##2)##3##4{\thinlines\unitlength1pt%
\lower 6.5pt\hbox{\begin{picture}(24,18)%
\put(##1){\vector(##2){24}}%
\put(0,8){$\scriptstyle##3$}%
\put(20,8){$\scriptstyle##4$}%
\end{picture}}}
\def\vspace##1{\noalign{\vskip##1\relax}}\relax%
\let\amp@rs@nd@&\iffalse}\fi
 \CD@true\vcenter\bgroup\relax\let\\=\cr\iffalse}\fi\tabskip\z@skip\baselineskip20\ex@
 &\hfill$\m@th##$\hfill\cr}
\def\@endCDS{\cr\egroup\egroup}
\begin{document}
\title{A Generalizationed theorem of Katz and motivic integration}
\author{Andrew Stout}
\address{Graduate Center, City  University of New York\\ 365 Fifth
Avenue\\10016\\U.S.A.}
\email{astout@gc.cuny.edu}

\date{May 1st, 2012}
\maketitle

\tableofcontents

\section*{Introduction}

In what follows, we are  interested in an extension of a theorem of Nicholas
Katz, which will be useful in studying the cohomology of generalized arc spaces
develop by Hans Schoutens in \cite{Sch1} and \cite{Sch2}. As is well known, one
is typically interested in the motivic volume of a definable subset of
$\mathcal{X} \times X \times \mathbb{Z}^n$ where $\mathcal{X}$ is a scheme over
$k((t))$ and $X$ the special fiber of $\mathcal{X}$, cf., \cite{CL1} . Schoutens
has introduced the possibility of developing a motivic integration for
\textit{limit points} other than  $k[[t]]$. In this note, we are concerned with
a special type of limit point $k[[T]]$ where $T=(t_i)_{i\in\mathbb{N}}$.

\section{Two lemmas}

We start with a few lemmas from commutative algebra which we will need. 

\begin{lemma} Let $R:=k[[x_1, x_2,\ldots, x_n, \ldots]]$ be the $\mathfrak{m}$-adic 
completion of
the polynomial ring $k[x_1,x_2,\ldots,x_n,\ldots]$ along the maximal ideal 
$\mathfrak{m}= (x_1,x_2,\ldots,x_n,\ldots)$. For all $n\in\mathbb{N}$, let $R_n
:= k[[x_1,\ldots,x_n]] \cong R/(x_{n+1},x_{n+2},\ldots)$  and let $R_n
\rightarrow R_{n-1}$ be the homomorphism with kernal $(x_n)R_n$. Then there is
an isomorphism $$R \cong \varprojlim_{n} R_n$$
Moreover, $R$ is a local ring with maximal ideal $\frak{m}R$. \end{lemma}

\begin{proof} It is straightforward to verify\footnote{You can do this by
verifying that  $\mathfrak{m}$ is the additive subgroup of non-units} that
$\mathfrak{m}R$ is the maximal ideal of $R$. For the other claim,
we define a homomorphism from $R$ to 
$ \varprojlim_{n} R_n$ by $x_i \mapsto (y_j)_{j\in\mathbb{N}}$ where $y_j = 0$
if $j< i$ and $y_j = x_i$ if $i\geq j$. By definition of inverse limit, this map
is injective.
 By the universal property of inverse limits, we conclude that it is
surjective. \end{proof}

\begin{rk}
 Note that this isomorphism takes place in the category of $k$-algebras and not the category of topological $k$-algebras. 
\end{rk}

Below we will state a version of Nakayama's Lemma which will be important for
our work below

\begin{lemma}  Let $R$ be any local ring (or, in particular, the one above), and
let $M$ be a finitely generated $R$-module. Then there is a surjective
homormophism of $R$-modules 
$$M/\mathfrak{m}M \otimes_{\mathbb{Z}} R \rightarrow M$$ \end{lemma}

\begin{proof}  This is a special case of  Proposition 2.6 of \cite{AM}
\end{proof}

\section{The theorem}
From now on, we assume that $k$ is of characteristic zero. 
 What follows is a natural extension, mutatis mutadis, of a theorem contained in
a paper of Katz, cf., Proposition 8.9 of \cite{Katz}.  The original argument is
originally due to Cartier, whereas my contribution is to show that it works with
an inverse system.

\begin{thm}
Let $M$ be finite $R$-module with a connection $\nabla$ arrising from the
continuous $k$-derivations coming from $R$ to $M$. Then $M^{\nabla}$ is finitely
generated and $$M \cong M^{\nabla}\otimes_k R$$
\end{thm}

\begin{proof} For all $i\in \mathbb{N}$ we define $$D_i =
\nabla(\frac{\partial}{\partial x_i})$$ and for each $j\in\mathbb{N}$ we define
$$D_{i}^{(j)} = \frac{1}{j!}(\nabla(\frac{\partial}{\partial x_i}))^j$$

For any $n$ and any $n$-tuple $J_n=(j_1,\ldots,j_n)\subset \mathbb{N}^n$, we
define the following 

$$ D^{J_n} = \Pi_{i=1}^{n} D_{i}^{j_i} \ \ \ x^{J_n} = \Pi_{i=1}^{n} x_{i}^{j_i}
\ \ \ (-1)^{J_n} = \Pi_{i=1}^{n} (-1)^{j_i}$$

Then, for each $n\in \mathbb{N}$ we successfully define an (additive)
endomorphism $P_n$ by $$P_n : M \rightarrow M,  \ \ \ P_n = \sum_{J_n}
(-1)^{J_n}x^{J_n}D^{J_n}$$

Now the action of $R$ on $M$  is actually the inverse limit homomorphisms
$\rho_n : R_n \rightarrow M$ of $k$-modules. In fact, $P_n$ will be considered
an additive endomorphism of $M$ as an $R_n$-module (via the isomorphism
established in Proposition 1). We define $P: M \rightarrow M$ to be the inverse
limit $$P=\varprojlim_{n} P_n$$ 
More explicitly, consider $f \in R$, which by Proposition 1, can be identified
with a sequence $(f_n)_{n\in\mathbb{N}}$ where $f_n\in R_n$, then $$P(f m) =
(P_n(f_n))P(m)$$
It is straightforward, that $P_n(f_n)=f_n(0) \ \forall n\in\mathbb{N}$, from
which it follows that for all $f \in R$ and all $m\in M$
$$P(fm)=f(0)P(m)$$
Therefore, the kernal of $P$ contains $\mathfrak{m}M$ where $\mathfrak{m}$ is
the maximal ideal of $R$. 

As we will now pass to the quotient $M/\mathfrak{m}M$, we mention that it is
not hard to see that the inverse system defined in Proposition 1 and hence above
satisfies the Mittag Leffler Condition. Therefore, there is an isomorphism
$$M/\mathfrak{m}M \cong \varprojlim_{n} M/(x_1,\ldots,x_n)M$$
Note that, for all $n$, $P_n$ induces the identity on $M/(x_1,\ldots,x_n)M$, 
and so $P$ induces the identity on $M/\mathfrak{m}M$ -- i.e., $$P(m) \equiv m
\mbox{ mod }\mathfrak{m}$$
Therefore, the kernel of $P$ is $\mathfrak{m}$. In a similar fashion it is easy
to check that $P$ as the following properties
$$ P|_{M^{\nabla}} = id_{M^{\nabla}} \ \ \  P(M) \subset M^{\nabla} \ \ \ P^2 =
P$$ 
Therefore, $P$ induces an isomorphism vector spaces over $k$
$$M/\mathfrak{m}M \cong M^{\nabla}$$
Therefore, $M^{\nabla}$ is a finite $R$-module. Using Nakayama's Lemma (see
Proposition 2 above), we have a surjective map $$M^{\nabla} \otimes_k R
\rightarrow M$$
Now, we will show that it is an isomorphism. Let $m_1,\ldots, m_l$ be
$k$-linearly independent elements of $M^{\nabla}$ and let $f_1,\ldots,f_l$ be
element of $R$, we need to show $$\sum_{k=1}^{l} f_k m_k \neq 0$$ In other
words, writing $f_k$ as its corresponding sequence $(f_{n}^{(k)})$ in the
inverse system, we need to show that for sufficiently large $n$ $$\sum_{k=1}^{l}
f_{n}^{(k)} m_k \neq 0 \ \ \ \ (*)$$
The only reason we to specify that $n$ be sufficiently large is to insure that
there is an $n$ so that $f_{n}^{(k)} \neq 0$ for some $k$, which is clearly
satisfied or else there is nothing to prove. Thus, we may assume there exists an
$N$ such that for all $n>N$ $$f_{n}^{(1)} \neq 0$$
Then for all $n\geq N$ there exists a $n$ tuple $J_n=(j_1,\ldots,j_n)$ such that
$$\Pi_{\nu=1}^{n} \frac{1}{j_{\nu}!}(\frac{\partial}{\partial
x_{\nu}})^{j_{\nu}}(f_{n}^{(1)})(0) \neq 0$$
Now, assume for the sake of contradiction that $$\sum_{k=1}^{l} f_{n}^{(k)} m_k
= 0$$ Applying $D^{J_n}$ to this equation, we get $$0 = D^{J_n}(\sum_{k=1}^{l}
f_{n}^{(k)} m_k ) = \sum_{k=1}^{l}\Pi_{\nu=1}^{n}
\frac{1}{j_{\nu}!}(\frac{\partial}{\partial
x_{\nu}})^{j_{\nu}}(f_{n}^{(i)})m_k$$ 
This is a sum of the form $$\sum_{k=1}^{l} g_k m_k =0, \ \ \ g_1(0)\neq 0, \ \ \
g_k \in R_n$$ 
Applying $P$ to this sum, we obtain $$\sum_{k=1}^{l}g_k(0)m_k = 0$$ which is
impossible as $g_1(0)\neq 0$ and the $m_1,\ldots, m_l$ are a $k$-linearly
independent set. Therefore, this must be an isomorphism. 
\end{proof}

\section{Application of result}

To apply the above theorem, we  take $M = H_{DR}^{\cdot}(X/S)$ to be
finite sheaf of modules on $S$, which is assured to us when we take $f : X
\rightarrow S$ to locally of finite type. We can define arc spaces by a
universal property:
we say that $T \rightarrow X$ is the arc space of $X$ along a scheme $Z$,
working in the category of $k$-schemes,  if for every closed fat point $\eta$ of
$T$ we have a unique morphism $\eta\times_k Z \rightarrow X$ making the
following diagram commute 
 $$\xymatrix{\eta \ar[d] \ar[r] & \eta \times_k Z \ar[d] \\ T \ar[r] &X }$$ 
and which is unique in the sense that if $T' \rightarrow X$ is any other such
space, we have a unique map $T' \rightarrow T$.  When such a scheme exists, we
write $\mathcal{A}_{Z}X$ for the arc space of $X$ along $Z$. This is a generalization
of the notion of arc space found in \cite{Sch2}.

Using this description of $\mathcal{A}_{Z} \Spec(k) = Z$  to conclude that $\mathcal{A}_{Z} X$
is a scheme over $Z$. In particular, if $\fx$ is a limit point (the direct limit of an infinite sequence of fat points), then we have
have the following relation

$$H_{DR}^{\cdot}(\mathcal{A}_{\fx} X/\fx)^{\nabla} \cong H_{DR}^{\cdot}(X/k)$$ when
$\nabla_{\fx} X \rightarrow \fx$ is smooth. This last condition implies, for
suitable point systems, that $X$ is rationally $\fx$-laxly stable -- cf.,
\cite{St}.  Therefore, we would expect a further decomposition of
$H_{DR}^{\cdot}(\mathcal{A}_{\fx} X/\fx)^{\nabla}$ which is captured motivically by
the rational motivic measure as displayed loc. cit.

\end{document}